\documentclass[12pt]{amsart}
\usepackage{amsthm,amsmath,amssymb,enumerate}
\newtheorem{thm}{Theorem}[section]
\newtheorem{prop}[thm]{Proposition}
\newtheorem{lemm}[thm]{Lemma}

\theoremstyle{definition}

\newcommand{\C}{\mathbb{C}}

\newcommand{\Q}{\mathbb{Q}}

\newcommand{\M}{\mathrm{M}}
\newcommand{\s}{\sigma}
\newcommand{\Irr}{\operatorname{Irr}}
\newcommand{\bs}{\boldsymbol}

\begin{document}
\title[A construction of pairs of association schemes]%
{A construction of pairs of non-commutative rank $8$
  association schemes from non-symmetric rank $3$ association schemes}
\author[A. Hanaki]{Akihide Hanaki}
\address[A. Hanaki]{Faculty of Science,
  Shinshu University, 3-1-1 Asahi,
  Matsumoto, 390-8621, Japan}
\email{hanaki@shinshu-u.ac.jp}

\author[M. Yoshikawa]{Masayoshi Yoshikawa}
\address[M. Yoshikawa]{Department of Mathematics,
  Hyogo University of Teacher Education, 942-1 Shimokume,
  Kato, Hyogo, 673-1494, Japan}
\email{myoshi@hyogo-u.ac.jp}

\keywords{character table; quaternion algebras; association schemes}

\subjclass[2010]{05E30} 
 
\thanks{Akihide Hanaki was supported by JSPS KAKENHI Grant Number JP17K05165.
  Masayoshi Yoshikawa was supported by JSPS KAKENHI Grant Number JP17K05173.}

\begin{abstract}
We construct a pair of non-commutative rank $8$ association schemes
from a rank $3$ non-symmetric association scheme.
For the pair, two association schemes have the same character table
but different Frobenius-Schur indicators.
This situation is similar to the pair of the dihedral group and the quaternion group of order $8$.
We also determine the structures of adjacency algebras of them over the rational number field.
\end{abstract}

\maketitle

\section{Introduction}
From a rank $3$ non-symmetric association scheme of order $n-1$,
we construct a pair of association schemes $(\mathcal{D}, \mathcal{Q})$ with the following properties.
\begin{itemize}
  \item $\mathcal{D}$ and $\mathcal{Q}$ are non-commutative, of order $4n$, and of rank $8$.
  \item $\mathcal{D}$ and $\mathcal{Q}$ have the identical character tables
  but their Frobenius-Schur indicators are different.
\end{itemize}
These properties are similar to the pair of the dihedral group $D_8$ and the quaternion group $Q_8$ of order $8$.

In theory of association schemes, many authors considered adjacency algebras over
the complex number field.
In this paper, we determine the structures of adjacency algebras of
$\mathcal{D}$ and $\mathcal{Q}$ over the rational number field $\Q$.
We prove 
\begin{eqnarray*}
  \mathbb{Q}\mathcal{D}&\cong& \mathbb{Q}\oplus\mathbb{Q}\oplus\mathbb{Q}\oplus\mathbb{Q}\oplus \M_2(\mathbb{Q}),\\
  \mathbb{Q}\mathcal{Q}&\cong& \mathbb{Q}\oplus\mathbb{Q}\oplus\mathbb{Q}\oplus\mathbb{Q}\oplus \mathbb{Q}(-1,-a),
\end{eqnarray*}
where $\M_2(\mathbb{Q})$ is the full matrix algebra of degree $2$ and $\mathbb{Q}(-1,-a)$
is the quaternion division algebra.

It is known that a rank $3$ non-symmetric association scheme of order $n-1$
exists if and only if there exists a skew-Hadamard matrix of order $n$ with $n\equiv 0 \pmod{4}$.
There is a conjecture that a skew-Hadamard matrix of order $n$ exists for an arbitrary $n\equiv 0\pmod{4}$.

\section{Preliminaries}
For a field $K$, we denote by $\M_n(K)$ the full matrix algebra of degree $n$ over $K$.
For a matrix $M$, the transposed matrix of $M$ will be denoted by ${}^tM$.
By $I_n$, we denote the identity matrix of degree $n$.
By $J_n$, we denote the $n\times n$ matrix with all entries $1$.

We define an association scheme in matrix form.
Let $\mathcal{S}=\{A_0,\dots,A_d\}$ be a set of non-zero $n\times n$ matrices with entries in $\{0,1\}$.
Then the set $\mathcal{S}$ is called an \emph{association scheme} of \emph{order} $n$ and \emph{rank} $d+1$ if
\begin{enumerate}
  \item $A_0=I_n$,
  \item $\sum_{i=0}^d A_i=J_n$, and
  \item for any $0\leq i, j\leq d$, ${}^tA_i$ and $A_iA_j$ are linear combinations of
  $A_0,\dots,A_d$.
\end{enumerate}
By definition, all rows of $A_i$ contain the same number of $1$.
We call this number the \emph{valency} of $A_i$ and denote it by $n_i$.
The association scheme $\mathcal{S}$ is said to be \emph{symmetric}
if ${}^tA_i=A_i$ for all $0\leq i\leq d$
and \emph{non-symmetric} otherwise.
The association scheme $\mathcal{S}$ is said to be \emph{commutative}
if $A_iA_j=A_jA_i$ for all $0\leq i, j\leq d$
and \emph{non-commutative} otherwise.
For a field $K$, $K\mathcal{S}:=\bigoplus_{i=0}^d KA_i$ is a $(d+1)$-dimensional $K$-algebra
by the condition (3).
We call $K\mathcal{S}$ the \emph{adjacency algebra} of $\mathcal{S}$ over $K$.
It is known that $K\mathcal{S}$ is semisimple if the characteristic of $K$ is $0$
\cite[Theorem 4.1.3 (ii)]{Zi}.
A \emph{representation} of $\mathcal{S}$ over $K$ means a $K$-algebra homomorphism
from $K\mathcal{S}$ to a full matrix algebra $\M_d(K)$ for some positive integer $d$.

Now we consider the adjacency algebra over the complex number field $\C$.
By Wedderburn's theorem \cite[3.5 Theorem]{Pierce},
$\C\mathcal{S}\cong \bigoplus_{i=1}^\ell \M_{d_i}(\C)$ for some $d_1,\dots,d_\ell$.
The set of projections $\C \mathcal{S}\to \M_{d_i}(\C)$ ($i=1,\dots,\ell$)
is a complete set of representatives of equivalence classes of irreducible representations of $\C\mathcal{S}$.
The matrix trace of a representation is called the \emph{character} of the representation.
By $\Irr(\mathcal{S})=\{\chi_1,\dots,\chi_\ell\}$, we denote the set of irreducible characters of $\C\mathcal{S}$.
The $\ell \times (d+1)$ matrix $(\chi_i(A_j))$
is called the \emph{character table} of $\mathcal{S}$.
Since $\C\mathcal{S}$ is defined as a matrix algebra, the map $A_i\mapsto A_i$ is a representation.
We call this the \emph{standard representation}
and its character the \emph{standard character} of $\mathcal{S}$.
Write the irreducible decomposition of the standard character by $\sum_{\chi\in\Irr(\mathcal{S})}m_\chi\chi$.
We call the coefficient $m_\chi$ the \emph{multiplicity} of $\chi$. 

The \emph{Frobenius-Schur indicator} $\nu(\chi)$ of $\chi\in\Irr(\mathcal{S})$ is defined by
$$\nu(\chi) := \frac{m_\chi}{n \chi(A_0)} \sum_{i=0}^d \frac{1}{n_i} \chi({A_i}^2).$$
An irreducible character is said to be of the \emph{first kind}
if it is afforded by a real representation, 
of the \emph{second kind} if it is afforded by a representation
which is equivalent to its complex conjugate but is not of the first kind,
and of the \emph{third kind} if it is not of the first or second kind.
The following theorem is known.

\begin{thm}\cite[(7,5), (7.6)]{Higman1975}\label{thm:FS}
  \begin{enumerate}
    \item For $\chi\in\Irr(\mathcal{S})$,
    $$\nu(\chi) = \begin{cases}
      1 & \text{if $\chi$ is of the first kind,} \\
      -1 & \text{if $\chi$ is of the second kind,} \\
      0 & \text{if $\chi$ is of the third kind.}\end{cases}$$
     \item $\sum_{\chi\in\Irr(\mathcal{S})}\nu(\chi)\chi(A_0)=\sharp\{A_i\in \mathcal{S}\mid {}^tA_i=A_i\}$.
  \end{enumerate}
\end{thm}

\section{Construction}
Let $\{A_0=I, A_1, A_2= {}^tA_1\}$ be a non-symmetric rank $3$ association scheme of order $n-1$.
Remark that $n\equiv 0\pmod{4}$ to exist such an association scheme.
We will construct a pair $(\mathcal{D}, \mathcal{Q})$ of
association schemes with the properties  described in Introduction.

Set $a := n-1$ and $b:=(n-2)/2$.

The following lemma is well-known.

\begin{lemm}\label{lem3.1}
  \begin{enumerate}
    \item $A_1^2=\frac{b-1}{2}A_1+\frac{b+1}{2}A_2$.
    \item $A_2^2=\frac{b+1}{2}A_1+\frac{b-1}{2}A_2$.
    \item $A_1A_2=A_2A_1=bA_0+\frac{b-1}{2}A_1+\frac{b-1}{2}A_2$.
  \end{enumerate}
\end{lemm}

Set $x:=(1,2,3,4)$, $y:=(1,2)(3,4)$, permutations of degree $4$,
and $G:=\langle x, y\rangle\cong D_8$.
We identify elements of $G$ with the corresponding permutation matrices,
namely
$$x = \left(\begin{array}{cccc}0&1&0&0\\0&0&1&0\\0&0&0&1\\1&0&0&0\end{array}\right),
\quad
y = \left(\begin{array}{cccc}0&1&0&0\\1&0&0&0\\0&0&0&1\\0&0&1&0\end{array}\right),$$
and so on.
The next lemma is clear by definition.

\begin{lemm}\label{lem3.2}
  As matrices, $1+x^2=xy+x^3y$.
\end{lemm}

We keep the above notations in the rest of this article.

\subsection{The association scheme $\mathcal{D}$}
We define $n\times n$ matrices by
$$E:=\left(\begin{array}{c|ccc}
  0&1&\dots&1\\ \hline 1&&&\\ \vdots&&O&\\ 1&&&
  \end{array}\right),\quad
  \tilde{A}_i:=\left(\begin{array}{c|ccc}
  0&0&\dots&0\\ \hline 0&&&\\ \vdots&&A_i&\\ 0&&&
  \end{array}\right)\quad(i=1,2).$$

By Lemma \ref{lem3.1} and direct calculations, we have the following lemma.

\begin{lemm}\label{lem3.3}
  \begin{enumerate}
    \item $E^2+\tilde{A}_1\tilde{A}_2+\tilde{A}_2\tilde{A}_1=aI_n+b\tilde{A}_1+b\tilde{A}_2$.
    \item $\tilde{A}_1^2+\tilde{A}_2^2=b\tilde{A}_1+b\tilde{A}_2$.
    \item $E\tilde{A}_1+\tilde{A}_2E=E\tilde{A}_2+\tilde{A}_1E=bE$.
  \end{enumerate}
\end{lemm}

We consider a subgroup $H:=C_G(y)=\{1,x^2,y,x^2y\}\cong C_2\times C_2$ of $G$.
It is easy to see that $\sum_{h\in H}h=\sum_{g\in G\setminus H}g =J_4$.
We define $4n\times 4n$ matrices by
$$\sigma_h:=I_n\otimes h \quad \text{for $h\in H$}$$
and
\begin{eqnarray*}
  \mu_g&:=&E\otimes g+\tilde{A_1}\otimes gy+\tilde{A_2}\otimes gx^2y\\
       &=&(I_n\otimes g)(E\otimes 1+\tilde{A_1}\otimes y+\tilde{A_2}\otimes x^2y) \\
       &=&(E\otimes 1+\tilde{A_1}\otimes x^2y+\tilde{A_2}\otimes y)(I_n\otimes g) \quad
           \text{for $g\in G\setminus H$.}
\end{eqnarray*}
We will show that $\mathcal{D}:=\{\sigma_h\mid h\in H\}\cup \{\mu_g\mid g\in G\setminus H\}$
forms an association scheme.

\begin{lemm}\label{lem3.4}
  The set $\mathcal{D}$ is closed by the transposition and
  $\sum_{h\in H}\sigma_h+\sum_{g\in G\setminus H}\mu_g=J_{4n}$.
\end{lemm}

\begin{proof}
  We have
  ${}^t\sigma_h=\sigma_h$ for $h\in H$,
  ${}^t\mu_{xy}=\mu_{xy}$,
  ${}^t\mu_{x^3y}=\mu_{x^3y}$
  and ${}^t\mu_{x}=\mu_{x^3}$.
  Thus $\mathcal{D}$ is closed by transposition.

  Since $\sum_{h\in H}h=\sum_{g\in G\setminus H}g=J_4$ and $E+\tilde{A}_1+\tilde{A}_2=J_n-I_n$,
  we have $\sum_{h\in H}\sigma_h+\sum_{g\in G\setminus H}\mu_g=J_{4n}$.
\end{proof}

\begin{lemm}\label{lem3.5}
  \begin{enumerate}
    \item $\sigma_h\sigma_{h'}=\sigma_{hh'}$ for $h, h'\in H$.
    \item $\sigma_h\mu_g=\mu_{hg}$ and $\mu_g\sigma_h=\mu_{gh}$ for $h\in H$ and $g\in G\setminus H$.
    \item $\mu_g\mu_{g'}=a\sigma_{gg'}+b\mu_{gg'x}+b\mu_{gg'x^3}$ for $g, g'\in G\setminus H$.
  \end{enumerate}
\end{lemm}

\begin{proof}
  It is easy to show that (1) and (2) hold.

  Suppose $g, g'\in G\setminus H$.
  Remark that $gg'\in H$.
  By $g'x^2=x^2g'$, $yg'=g'x^2y$, Lemma \ref{lem3.2} and Lemma \ref{lem3.3}, we have
  \begin{eqnarray*}
    \mu_g\mu_{g'}&=& (I_n\otimes g)(E\otimes 1+\tilde{A_1}\otimes y+\tilde{A_2}\otimes x^2y)
                (E\otimes 1+\tilde{A_1}\otimes x^2y+\tilde{A_2}\otimes y)(I_n\otimes g')\\
                &=& (I_n\otimes g)(aI_n\otimes 1+b(\tilde{A}_1+\tilde{A}_2)\otimes(1+x^2)
                    +bE\otimes (y+x^2y))(I_n\otimes g')\\
                 &=& a\sigma_{gg'}+b(\tilde{A}_1+\tilde{A}_2)\otimes g(1+x^2)g'
                    +bE\otimes g(y+x^2y)g'\\
                 &=& a\sigma_{gg'}+b(\tilde{A}_1+\tilde{A}_2)\otimes gg'(1+x^2)
                    +bE\otimes gg'x^2(y+x^2y)\\
                 &=& a\sigma_{gg'}+b(\tilde{A}_1+\tilde{A}_2)\otimes gg'(xy+x^3y)
                    +bE\otimes gg'x^2y(xy+x^3y)\\
                 &=& a\sigma_{gg'}+b(\tilde{A}_1+\tilde{A}_2)\otimes gg'(xy+x^3y)
                    +bE\otimes gg'(x+x^3)\\
                 &=& a\sigma_{gg'}+b\mu_{gg'x}+b\mu_{gg'x^3}.
  \end{eqnarray*}
  Now (3) holds.
\end{proof}

\begin{thm}\label{thm3.6}
  The set $\mathcal{D}=\{\sigma_h\mid h\in H\}\cup \{\mu_g\mid g\in G\setminus H\}$ forms an association scheme.
\end{thm}

\begin{proof}
  This is clear by Lemma \ref{lem3.4} and Lemma \ref{lem3.5}.
\end{proof}

\subsection{The association scheme $\mathcal{Q}$}
We define $n\times n$ matrices by
$$\hat{A}_1:=\left(\begin{array}{c|ccc}
  0&1&\dots&1\\ \hline 0&&&\\ \vdots&&A_1&\\ 0&&&
  \end{array}\right),\quad
  \hat{A}_2:=\left(\begin{array}{c|ccc}
  0&0&\dots&0\\ \hline 1&&&\\ \vdots&&A_2&\\ 1&&&
  \end{array}\right).$$

By Lemma \ref{lem3.1} and direct calculations, we have the following lemma.

\begin{lemm}\label{lem3.7}
  \begin{enumerate}
    \item $\hat{A}_1^2+\hat{A}_2^2=b(J_n-I_n)$.
    \item $\hat{A}_1\hat{A}_2+\hat{A}_2\hat{A}_1=bJ_n+(a-b)I_n$.
  \end{enumerate}
\end{lemm}

We consider a subgroup $K:=C_G(x)=\{1,x,x^2,x^3\}\cong C_4$ of $G$.
It is easy to see that $\sum_{k\in K}k=\sum_{g\in G\setminus K}g =J_4$.
We define $4n\times 4n$ matrices by
$$\sigma_k:=I_n\otimes k \quad \text{for $k\in K$}$$
and
\begin{eqnarray*}
  \tau_g&:=&\hat{A_1}\otimes g+\hat{A_2}\otimes gx^2\\
       &=&(I_n\otimes g)(\hat{A_1}\otimes 1+\hat{A_2}\otimes x^2) \\
       &=&(\hat{A_1}\otimes 1+\hat{A_2}\otimes x^2)(I_n\otimes g) \quad
           \text{for $g\in G\setminus K$.}
\end{eqnarray*}
We will show that $\mathcal{Q}:=\{\sigma_k\mid k\in K\}\cup \{\tau_g\mid g\in G\setminus K\}$ forms an association scheme.

\begin{lemm}\label{lem3.8}
  The set $\mathcal{Q}$ is closed by the transposition and
  $\sum_{k\in K}\sigma_k+\sum_{g\in G\setminus K}\tau_g=J_{4n}$.
\end{lemm}

\begin{proof}
  We have
  ${}^t\sigma_1=\sigma_1$,
  ${}^t\sigma_{x^2}=\sigma_{x^2}$,
  ${}^t\sigma_x=\sigma_{x^3}$,
  ${}^t\tau_{y}=\tau_{x^2y}$,
  and ${}^t\tau_{xy}=\tau_{x^3y}$.
  Thus $\mathcal{D}$ is closed by transposition.

  Since $\sum_{k\in K}k=\sum_{g\in G\setminus K}g=J_4$ and $I_{n}+\hat{A}_1+\hat{A}_2=J_n$,
  we have $\sum_{k\in K}\sigma_k+\sum_{g\in G\setminus K}\tau_g=J_{4n}$.
\end{proof}

\begin{lemm}\label{lem3.9}
  \begin{enumerate}
    \item $\sigma_k\sigma_{k'}=\sigma_{kk'}$ for $k, k'\in K$.
    \item $\sigma_k\tau_g=\tau_{kg}$ and $\tau_g\sigma_k=\tau_{gk}$ for $k\in K$ and $g\in G\setminus K$.
    \item $\tau_g\tau_{g'}=a\sigma_{gg'x^2}+b\tau_{gg'xy}+b\tau_{gg'x^3y}$ for $g, g'\in G\setminus K$.
  \end{enumerate}
\end{lemm}

\begin{proof}
  It is easy to show that (1) and (2) hold.

  Suppose $g, g'\in G\setminus K$.
  Remark that $gg'\in K$.
  By Lemma \ref{lem3.2} and Lemma \ref{lem3.7}, we have
  \begin{eqnarray*}
    \tau_g\tau_{g'}&=& (I_n\otimes g)(\hat{A_1}\otimes 1+\hat{A_2}\otimes x^2)^2(I_n\otimes g')\\
                     &=& (I_n\otimes gg')(\hat{A_1}^2\otimes 1+(\hat{A_1}\hat{A_2}+\hat{A_2}\hat{A_1})\otimes x^2
                         +\hat{A_1}^2\otimes 1)\\
                     &=& (I_n\otimes gg')(b(J_n-I_n)\otimes 1+(bJ_n+(a-b)I_n)\otimes x^2)\\
                     &=& (I_n\otimes gg')(aI_n\otimes x^2+b(J_n-I_n)\otimes(1+x^2))\\
                     &=& (I_n\otimes gg')(aI_n\otimes x^2+b(\hat{A_1}+\hat{A_2})\otimes(1+x^2))\\
                     &=& (I_n\otimes gg')(aI_n\otimes x^2+b(\hat{A_1}+\hat{A_2})\otimes(xy+x^3y))\\
                     &=& a\sigma_{gg'x^2}+b\tau_{gg'xy}+b\tau_{gg'x^3y}.    
  \end{eqnarray*}
  Now (3) holds.
\end{proof}

\begin{thm}\label{thm3.10}
  The set $\mathcal{Q}=\{\sigma_k\mid k\in K\}\cup \{\tau_g\mid g\in G\setminus K\}$ forms an association scheme.
\end{thm}

\begin{proof}
  This is clear by Lemma \ref{lem3.8} and Lemma \ref{lem3.9}.
\end{proof}

\section{The character tables of $\mathcal{D}$ and $\mathcal{Q}$}
In this section, we will determine the character tables of $\mathcal{D}$ and $\mathcal{Q}$.
Consequently, we can see that the tables are identical.
Moreover, we will show that their Frobenius-Schur indicators are different.

In this section, we use some terminologies not defined in Preliminaries.
For them, the reader is referred to \cite{Zi} or \cite{Hanaki2005}, for example.

\begin{prop}\label{prop:charD}
  The character table of $\mathcal{D}$ is 
  $$\begin{array}{c|rrrrrrrr|c}
      &\s_1&\s_{x^2}&\s_y&\s_{x^2y}&\mu_x&\mu_{x^3}&\mu_{xy}&\mu_{x^3y}&m_{\chi_i}\\
      \hline
      \chi_1&1&1&1&1&a&a&a&a&1\\
      \chi_2&1&1&-1&-1&a&a&-a&-a&1\\
      \chi_3&1&1&1&1&-1&-1&-1&-1&a\\
      \chi_4&1&1&-1&-1&-1&-1&1&1&a\\
      \chi_5&2&-2&0&0&0&0&0&0&n\\
  \end{array}$$
  The Frobenius-Schur indicators are 
  $\nu(\chi_i)=1$ ($i=1,2,3,4, 5$).
\end{prop}

\begin{proof}
  By Lemma \ref{lem3.5}, $\mathcal{T}=\{\s_1,\s_{x^2},\s_y,\s_{x^2y}\}$ is a normal closed subset of $\mathcal{D}$.
  By \cite[Theorem 3.5]{Hanaki2003}, we can determine $\chi_1$ and $\chi_3$.
  Again, by Lemma \ref{lem3.5},
  $\mathcal{U}=\{\s_1,\s_{x^2},\mu_{x},\mu_{x^3}\}$ is a strongly normal closed subset of $\mathcal{D}$
  and $\chi_2$ is determined.
  By \cite[Theorem 3.5]{Hanaki2005}, $\chi_4:=\chi_2\chi_3$ is an irreducible character.
  Now, by $\sum_{i=1}^5 m_{\chi_i}\chi_i(\rho)=0$ for $\s_1\ne \rho\in \mathcal{D}$ \cite[Chap. 4]{Zi},
  we can determine $\chi_5$.

  Since $\chi_i$ ($i=1,2,3,4$) are rational characters of degree $1$, they are of the first kind.
  There are $6$ symmetric relations by the proof of Lemma \ref{lem3.4}.
  By Theorem \ref{thm:FS} (2),
  $$6=\sum_{i=1}^5 \nu(\chi_i)\chi_i(1)=1+1+1+1+2\nu(\chi_5)$$
  and we have $\nu(\chi_5)=1$.
\end{proof}

\begin{prop}\label{prop:charQ}
  The character table of $\mathcal{Q}$ is 
  $$\begin{array}{c|rrrrrrrr|c}
      &\s_1&\s_{x^2}&\s_x&\s_{x^3}&\tau_{xy}&\tau_{x^3y}&\tau_{y}&\tau_{x^2y}&m_{\varphi_i}\\
      \hline
      \varphi_1&1&1&1&1&a&a&a&a&1\\
      \varphi_2&1&1&-1&-1&a&a&-a&-a&1\\
      \varphi_3&1&1&1&1&-1&-1&-1&-1&a\\
      \varphi_4&1&1&-1&-1&-1&-1&1&1&a\\
      \varphi_5&2&-2&0&0&0&0&0&0&n\\
  \end{array}$$
  The Frobenius-Schur indicators are 
  $\nu(\varphi_i)=1$ ($i=1,2,3,4$) and $\nu(\varphi_5)=-1$.
\end{prop}

\begin{proof}
  By Lemma \ref{lem3.9}, $\mathcal{T}'=\{\s_1,\s_{x^2},\s_{x},\s_{x^3}\}$ is a normal closed subset of $\mathcal{Q}$.
  We can determine $\varphi_1$ and $\varphi_3$.
  Again, by Lemma \ref{lem3.9},
  $\mathcal{U}'=\{\s_1,\s_{x^2},\tau_{xy},\tau_{x^3y}\}$ is a strongly normal closed subset of $\mathcal{Q}$
  and $\varphi_2$ is determined.
  Similarly to Proposition \ref{prop:charD},
  We can determine the table.

  Since $\varphi_i$ ($i=1,2,3,4$) are rational characters of degree $1$, they are of the first kind.
  There are $2$ symmetric relations by the proof of Lemma \ref{lem3.8}.
  By Theorem \ref{thm:FS} (2),
  $$2=\sum_{i=1}^5 \nu(\varphi_i)\varphi_i(1)=1+1+1+1+2\nu(\varphi_5)$$
  and we have $\nu(\varphi_5)=-1$.
\end{proof}

\section{Irreducible representations and rational adjacency algebras}
We will determine irreducible representations and the structures of rational adjacency algebras
of $\mathcal{D}$ and $\mathcal{Q}$, respectively.

\begin{prop}\label{repD}
  The map $T:\mathcal{D}\to \M_2(\C)$ defined by
  $$\sigma_{x^2} \mapsto \left(\begin{array}{rr}-1&0\\0&-1\end{array}\right),\quad
  \sigma_{y} \mapsto \left(\begin{array}{rr}1&0\\0&-1\end{array}\right),\quad
  \mu_{x} \mapsto \left(\begin{array}{rr}0&-1\\a&0\end{array}\right)$$
  is an irreducible representation of $\mathcal{D}$ affording $\chi_5$.
\end{prop}

\begin{proof}
  We can check all products in Lemma \ref{lem3.5}.
\end{proof}

Similarly, we have the following proposition.

\begin{prop}\label{repQ}
  The map $T':\mathcal{Q}\to \M_2(\C)$ defined by
  $$\sigma_{x} \mapsto \left(\begin{array}{rr}\sqrt{-1}&0\\0&-\sqrt{-1}\end{array}\right),\quad
  \tau_{y} \mapsto \left(\begin{array}{rr}0&-1\\a&0\end{array}\right)$$
  is an irreducible representation of $\mathcal{Q}$ affording $\varphi_5$.
\end{prop}

Now we can determine the structures of rational adjacency algebras.
To describe the result, we define (generalized) quaternion algebras \cite[1.6]{Pierce}.
For a field $F$ and $r,s\in F\setminus\{0\}$, a \emph{quaternion algebra} $F(r,s)$ is a four dimensional
$F$-algebra with basis $1$, $\bs{i}$, $\bs{j}$, and $\bs{k}$ with the products
$$\bs{i}^2=r,\quad \bs{j}^2=s,\quad \bs{i}\bs{j}=-\bs{j}\bs{i}=\bs{k}.$$
If $r$ and $s$ are negative rational numbers, then $\Q(r,s)$ is a division algebra.

\begin{prop}
  As $\Q$-algebras, we have the following isomorphisms.
  \begin{enumerate}
    \item $\Q\mathcal{D}\cong \Q\oplus\Q\oplus\Q\oplus\Q\oplus \M_2(\Q)$.
    \item $\Q\mathcal{Q}\cong \Q\oplus\Q\oplus\Q\oplus\Q\oplus \Q(-1,-a)$.
    ($\Q(-1,-a)$ is a division algebra.)
  \end{enumerate}
\end{prop}

\begin{proof}
  Since $\chi_i$ ($i=1,2,3,4$) have rational values on $\mathcal{D}$, we can determine
  $\Q\oplus\Q\oplus\Q\oplus\Q$.
  For the representation $T$ in Proposition \ref{repD},
  it is easy to see that $T(\Q\mathcal{D})=\M_2(\Q)$.
  Thus (1) holds.

  Since $\varphi_i$ ($i=1,2,3,4$) have rational values on $\mathcal{Q}$, we can determine
  $\Q\oplus\Q\oplus\Q\oplus\Q$.
  For the representation $T'$ in Proposition \ref{repQ}, the set consisting of 
  \begin{eqnarray*}
  T'(\s_1)&=&\left(\begin{array}{rr}1&0\\0&1\end{array}\right),\quad 
  T'(\s_x)=\left(\begin{array}{rr}\sqrt{-1}&0\\0&-\sqrt{-1}\end{array}\right),\\
  T'(\tau_y)&=&\left(\begin{array}{rr}0&-1\\a&0\end{array}\right),\quad 
  T'(\tau_{xy})=\left(\begin{array}{rr}0&-\sqrt{-1}\\-a\sqrt{-1}&0\end{array}\right)
  \end{eqnarray*}
  is a $\Q$-basis of $T'(\Q\mathcal{Q})$.
  We can see that
  \begin{eqnarray*}
    && T'(\s_x)^2=-T'(\s_1), \quad T'(\tau_y)^2=-aT'(\s_1),\\
    && T'(\s_x)T'(\tau_{y})=-T'(\tau_{y})T'(\s_x)=T'(\tau_{xy}).
  \end{eqnarray*}
  This shows that $T'(\Q\mathcal{Q})\cong \Q(-1,-a)$ and (2) holds.
\end{proof}

\section{Remark}
Our pair $(\mathcal{D}, \mathcal{Q})$ has similar properties to the pair $(D_8, Q_8)$,
the dihedral group and quaternion group of order $8$.
Our association schemes have order $4n$ and $n\equiv 0\pmod{4}$,
and thus $(D_8,Q_8)$ is not obtained by our construction.
If we set $a=1$ and $A_1=A_2=O$ and apply our construction,
then we can construct the pair $(D_8,Q_8)$ and all arguments are valid for it.

\bibliographystyle{amsplain}
\providecommand{\bysame}{\leavevmode\hbox to3em{\hrulefill}\thinspace}
\providecommand{\MR}{\relax\ifhmode\unskip\space\fi MR }
\providecommand{\MRhref}[2]{%
  \href{http://www.ams.org/mathscinet-getitem?mr=#1}{#2}
}
\providecommand{\href}[2]{#2}

\end{document}